\documentclass[11pt,twoside]{amsart}

\usepackage{amssymb}
\usepackage{hyperref}

\numberwithin{equation}{section}
\newtheorem{theorem}{Theorem}[section]
\newtheorem{lemma}[theorem]{Lemma}

\newtheorem{conjecture}[theorem]{Conjecture}

\theoremstyle{definition}
\newtheorem{definition}[theorem]{Definition}
\newtheorem{remark}[theorem]{Remark}
\newtheorem{example}[theorem]{Example}

\newcommand{\popo}{\mathbb{P}^1 \times \mathbb{P}^1}

\newcommand{\pr}{\mathbb{P}}

\begin{document}


\title{Fat lines in $\mathbb{P}^3$: powers versus symbolic powers}

\author{Elena Guardo}
\address{Dipartimento di Matematica e Informatica\\
Viale A. Doria, 6 - 95100 - Catania, Italy}
\email{guardo@dmi.unict.it}
\urladdr{http://www.dmi.unict.it/$\sim$guardo/}

\author{Brian Harbourne}
\address{Department of Mathematics \\
University of Nebraska--Lincoln\\
Lincoln, NE 68588-0130, USA }
\email{bharbour@math.unl.edu}
\urladdr{http://www.math.unl.edu/$\sim$bharbourne1/}

\author{Adam Van Tuyl}
\address{Department of Mathematical Sciences \\
Lakehead University \\
Thunder Bay, ON P7B 5E1, Canada}
\email{avantuyl@lakeheadu.ca}
\urladdr{http://flash.lakeheadu.ca/$\sim$avantuyl/}

\keywords{symbolic powers, points, lines, multigraded}
\subjclass[2000]{13F20, 13A15, 14C20}
\date{August 22, 2012}

\begin{abstract}
We study the symbolic and regular powers of
ideals $I$ for a family of special configurations of lines in $\pr^3$.
For this family, we show
that $I^{(m)} = I^m$ for all integers $m$ if and only if
$I^{(3)} = I^3$.  We use these configurations to answer a question
of Huneke that asks whether $I^{(m)} = I^{m}$ for
all $m$ if equality holds when $m$ equals the big height
of the ideal $I$.
\end{abstract}

\maketitle

\section{Introduction}
Let $R = k[x_0,\ldots,x_N] = k[\mathbb{P}^N]$ be a polynomial ring over an
algebraically closed field of characteristic zero.  Starting with the
work of \cite{ELS,HH}, and further refined
by \cite{BCH,BH,BH2,DJ,GHM,GHVT,GHVT2,HaHu,HH2,LS, Mo},
the following
containment question has been of interest:
given a homogeneous ideal $(0) \neq I \subsetneq R$, for what integers
$m$ and $r$ do we have $I^{(m)} \subseteq I^r$?  Here, $I^{(m)}$ denotes
the {\it $m$-th symbolic power} of the ideal $I$
whose definition we now recall. If $I^m = \bigcap_j Q_j$ is the
homogeneous primary decomposition of $I^m$,
then $I^{(m)} = \bigcap_{i}
Q'_i$, where the intersection is over those primary components $Q'_i$ which have
$\sqrt{Q'_i}$ contained in an associated prime ideal of $I$.  From the
definition, we always have $I^m \subseteq I^{(m)}$, but we do not always have
$I^{(m)} \subseteq I^m$. For non-trivial ideals we never have $I^{(m)} \subseteq I^r$
when $m<r$, but by \cite{ELS,HH} we always have $I^{(m)} \subseteq I^r$
when $m\geq rN$;  in fact,
$I^{(me)} \subseteq I^m$ where $e$ is the big height of $I$ (i.e., the height of
the associated prime ideal of $I$ of biggest height).
Consequently, for each fixed $r$, it is of interest to find the smallest
$m\geq r$ with $I^{(m)} \subseteq I^r$.
Apart from some sporadic examples \cite{GHVT,LS},
most of the cases for which this smallest $m$ are known either are ideals
of complete intersections (i.e., ideals $I$ generated by a regular sequence,
in which case $I^{(m)}=I^m$ for all $m\geq 1$ \cite{ZS}) or are ideals
defining zero-dimensional schemes \cite{BH2}.

Thus constructing families of ideals $I$ of positive dimensional schemes which are not
complete intersections but for which we can determine the least $m$ for each $r$ such that
$I^{(m)} \subseteq I^r$ is of particular interest. Our focus will be on
ideals with extremal behavior, in the sense that they
satisfy $I^{(m)} = I^m$ for all $m \geq 1$. Adding to the interest of our results
is that as an application we answer a question raised by C.\ Huneke.\footnote{This question was
posed by Huneke in a talk ``Comparing Powers and Symbolic Powers of Ideals''
that he gave at the University of Nebraska, Lincoln in May 2008.} In
particular, for a homogeneous ideal $I$ of big height
$c$, Huneke asked whether it is true that $I^{(m)} = I^m$ for all
$m \geq 1$ if $I^{(c)} = I^c$. Our results show that the answer in general is no.
The ideals that we look at here are ideals $I$ of special configurations of lines in
$\mathbb{P}^3$. For these configurations, we completely characterize when
$I$ satisfies $I^{(m)} = I^m$ for all $m \geq 1$.

The key idea behind our constructions is to build finite configurations
of lines in $\pr^3$ so that the associated ideals have a bigraded
structure with respect to which the ideals define finite sets of points in $\popo$.
Previous work on points in $\popo$, such as, for example, \cite{GuMaRa,GVT,GVT1,VT2},
then provides us with tools which we can exploit in our study of ideals of lines in $\pr^3$.
We now explain these two points of view.
Let $R =k[\pr^3] = k[x_0,x_1,y_0,y_1]$ and consider the two skew lines
$L_1$ and $L_2$ in $\pr^3$ defined by $I(L_1) = (x_0,x_1)$ and
$I(L_2) = (y_0,y_1)$.   If $B = [0:0:b_0:b_1]$ is a
point on $L_1$ and $A = [a_0:a_1:0:0]$ is a point on $L_2$, then the
line $L$ through $A$ and $B$ has defining ideal $I(L) =
(a_1x_0-a_0x_1,b_1y_0-b_0y_1)$.  We can regard $R$ as being
$k[\popo]$, and thus endowed with an
$\mathbb{N}^2$-graded structure by setting $\deg x_i = (1,0)$ and
$\deg y_i = (0,1)$.
With respect to this grading, $I(L)$ is a bigraded ideal which
defines the point $(P,Q)\in\popo$ where $P=[a_0:a_1]$ and $Q=[b_0:b_1]$
and thus $I(L)=I((P,Q))$.
Our configurations will be finite unions of such
lines, i.e., each line $L$ in our configuration meets both the
lines $L_1$ and $L_2$.  This allows us to reinterpret our union of
lines in $\mathbb{P}^3$ as a finite set of points $X$ in $\popo$.
Conversely, the ideal of every finite set of points in $\popo$
defines a finite union of lines in $\pr^3$ where every line intersects $L_1$ and $L_2$.
In addition, we will require that our configurations be arithmetically
Cohen-Macaulay (ACM), i.e., that their associated coordinate rings be
Cohen-Macaulay. Expressed in the language of points in $\popo$,
our main result is:

\begin{theorem}\label{ACMresult}
Let $I = I(X)$ be the ideal of a finite reduced ACM subscheme $X$ in
$\popo$.  Then $I^{(m)}=I^m$ for all $m\geq1$ if and only if
$I^{(3)}=I^3$.
\end{theorem}

To prove Theorem \ref{ACMresult}, we use a result of Morey \cite{Mo}
to first show that $I^{(m)} = I^m$ for all $m \geq 1$ if and
only if equality holds for $m=2$ and $3$.
We use results of the first and third author \cite{GVT1}
to show that $I^{(2)} = I^2$ always holds for ideals of a finite reduced
ACM subscheme in $\popo$, and thus $I^{(m)} = I^m$ for all $m \geq 1$ if and only if
$I^{(3)} = I^3$.
We note that $I^{(3)} \neq I^3$ and $I^{(3)} = I^3$
both occur; see Example \ref{sixptsexample} and Theorem \ref{nonacmexample}
for examples of the former, and Example \ref{3acm}
and Theorem \ref{examplesofI^3} for examples of the latter.
Note that the examples with $I^{(3)} \neq I^3$ give a negative answer to Huneke's question.
To see why, note that the ideal $I$ is an unmixed ideal of height two
(in particular, its big height is two) that has $I^{(2)} = I^2$,
but fails to have $I^{(m)} \neq I^m$ for all $m \geq 1$
since it fails for $m=3$.

In section 2 we present the background needed for the proof
of Theorem \ref{ACMresult} in section 3.
The last section presents a conjecture, along with some
evidence, for a geometric description of all finite reduced ACM
subschemes $X$ in $\popo$ with $I(X)^{(3)} = I(X)^3$.

{\bf Acknowledgements.} Some of the work on this paper was carried out
during visits by the first author at the University of Nebraska and during visits
by the second and third authors at the Universit\`a di
Catania; all of the authors thank both institutions for their hospitality.
The second author's work on this project
was sponsored by the National Security Agency under Grant/Cooperative
agreement ``Advances on Fat Points and Symbolic Powers,'' Number H98230-11-1-0139.
The United States Government is authorized to reproduce and distribute reprints
notwithstanding any copyright notice.
The third author acknowledges the financial support of NSERC and GNSAGA.


\section{Background}

To any nonempty finite set $X\subset\popo$ of points
we associate a set $S_X$ of integer lattice points
indicating which points lie on the same horizontal or vertical rule.
The idea is to enumerate the horizontal and vertical rules
whose intersection with $X$
is non-empty. We thus obtain, say, $H_1,\ldots,H_h$ and $V_1,\ldots,V_v$
where $X\subset \bigcup H_i$ and $X\subset \bigcup V_j$, and $S_X$
consists of all pairs $(i,j)$ such that $X \cap H_i\cap V_j\neq\varnothing$.
We also associate to $X$ its bi-homogeneous ideal
$I(X)=\bigcap_{(P,Q)\in X}I((P,Q))\subset R = k[x_0,x_1,y_0,y_1]$ and its coordinate
ring $R/I(X)$.

As is usual, we say that a subscheme $X \subseteq \popo$ is
arithmetically Cohen-Macaulay (ACM)
if its coordinate ring $R/I(X)$ is Cohen-Macaulay.
It is important to emphasize that while the coordinate rings
of zero-dimensional subschemes in $\mathbb{P}^n$ are always
Cohen-Macaulay, coordinate rings of zero-dimensional subschemes
in a multiprojective space need not be Cohen-Macaulay.
As an example of this, take two distinct points $P$
and $Q$ in $\pr^1$.  Then the coordinate ring $R/I(X)$ of
$X = \{(P,P), (Q,Q)\} \subseteq \popo$ is not Cohen-Macaulay. This can be seen
from the fact that if we consider only the graded structure of
$R/I(X)$, then $X$ is the union of two skew lines in
$\mathbb{P}^3$ which is well known not to be Cohen-Macaulay.
Reduced zero-dimensional
ACM subschemes of $\popo$ can be characterized
in terms of their Hilbert functions \cite{GuMaRa}. We recall
an alternative geometric characterization found in \cite[Theorem 4.3]{GVTcollect}:

\begin{theorem}\label{ACMcharacterize}
Let $X$ be the reduced subscheme consisting of a finite set of points in $\popo$.  Then
$X$ is ACM if and only if whenever $(P_1,Q_1)$ and $(P_2,Q_2)$ are
points in $X$ with $P_1 \neq P_2$ and $Q_1 \neq Q_2$, then either
$(P_1,Q_2)$ or $(P_2,Q_1)$ (or both) also belongs to $X$.
\end{theorem}

When a finite reduced subscheme $X \subseteq \popo$ is ACM,
Theorem \ref{ACMcharacterize} implies that  we can relabel the $H_i$'s and $V_j$'s
so that $S_X$ resembles the Ferrers diagram of a partition
$\lambda = (\lambda_1,\ldots,\lambda_h)$ with $\sum \lambda_i = |X|$ and
 $\lambda_i \geq \lambda_{i+1}$ for all $1\leq i<h$, where $\lambda_i$ equals the number of points
on the rule $H_i$ for each $i$.

\begin{example}\label{renumber}
Any finite reduced subscheme $X\subset\popo$ whose
diagram $S_X$ is as in Figure 1 is ACM.
This is because after relabelling the $H_i$'s and $V_j$'s, as in Figure 2,
$S_X$ becomes the Ferrers diagram
for the partition $\lambda = (6,5,3,1,1)$ (where
the entries of $\lambda$ count the number of
points on each $H_i$, arranged to be non-increasing).
\vskip\baselineskip

\hspace{.25in}\hbox to 3in{
\begin{picture}(150,95)(25,-10)
\put(-30,30){Figure 1}
\put(60,-10){\line(0,1){90}}
\put(80,-10){\line(0,1){90}} \put(100,-10){\line(0,1){90}}
\put(120,-10){\line(0,1){90}} \put(140,-10){\line(0,1){90}}
\put(160,-10){\line(0,1){90}}

\put(54,85){$V_1$}
\put(74,85){$V_2$}
\put(94,85){$V_3$}
\put(114,85){$V_4$}
\put(134,85){$V_5$}
\put(154,85){$V_6$}

\put(55,-5){\line(1,0){115}}
\put(55,15){\line(1,0){115}} \put(55,35){\line(1,0){115}}
\put(55,55){\line(1,0){115}} \put(55,75){\line(1,0){115}}

\put(35,-11){$H_5$}
\put(35,11){$H_4$}
\put(35,31){$H_3$}
\put(35,51){$H_2$}
\put(35,71){$H_1$}

\put(60,35){\circle*{5}}
\put(120,-5){\circle*{5}}

\put(80,35){\circle*{5}} \put(80,55){\circle*{5}}
\put(80,75){\circle*{5}}

\put(100,35){\circle*{5}} \put(100,55){\circle*{5}}

\put(120,15){\circle*{5}} \put(120,35){\circle*{5}}
\put(120,55){\circle*{5}} \put(120,75){\circle*{5}}

\put(140,55){\circle*{5}} \put(140,35){\circle*{5}}

\put(160,35){\circle*{5}} \put(160,55){\circle*{5}}
\put(160,75){\circle*{5}}
\end{picture}
}

\vskip-1.95in\vskip2\baselineskip
\hspace{3.75in}\hbox to 3in{
\begin{picture}(150,110)(25,-10)
\put(-30,30){Figure 2}
\put(60,-10){\line(0,1){90}}
\put(80,-10){\line(0,1){90}} \put(100,-10){\line(0,1){90}}
\put(120,-10){\line(0,1){90}} \put(140,-10){\line(0,1){90}}
\put(160,-10){\line(0,1){90}}

\put(54,90){$V_1$}
\put(74,90){$V_2$}
\put(94,90){$V_3$}
\put(114,90){$V_4$}
\put(134,90){$V_5$}
\put(154,90){$V_6$}

\put(55,-5){\line(1,0){115}}
\put(55,15){\line(1,0){115}} \put(55,35){\line(1,0){115}}
\put(55,55){\line(1,0){115}} \put(55,75){\line(1,0){115}}

\put(35,-9){$H_5$}
\put(35,11){$H_4$}
\put(35,31){$H_3$}
\put(35,51){$H_2$}
\put(35,71){$H_1$}

\put(60,-5){\circle*{5}} \put(63,-2){$$}
\put(60,15){\circle*{5}} \put(63,18){$$}
\put(60,35){\circle*{5}}
\put(63,38){$$} \put(60,55){\circle*{5}} \put(63,58){$$}
\put(60,75){\circle*{5}} \put(63,78){$$}

\put(80,55){\circle*{5}} \put(83,58){$$} \put(80,35){\circle*{5}}
\put(83,38){$$} \put(80,75){\circle*{5}} \put(83,78){$$}

\put(100,35){\circle*{5}} \put(103,38){$$}
\put(100,55){\circle*{5}} \put(103,58){$$}
\put(100,75){\circle*{5}} \put(103,78){$$}

\put(120,75){\circle*{5}} \put(123,78){$$}
\put(120,55){\circle*{5}} \put(123,58){$$}

\put(140,55){\circle*{5}} \put(143,58){$$}
\put(140,75){\circle*{5}} \put(143,78){$$}

\put(160,75){\circle*{5}} \put(163,78){$$}
\end{picture}
}
\end{example}

When $X \subseteq \popo$ is a finite reduced ACM subscheme, then the
generators of $I(Z)$ can be described in terms of the partition
$\lambda$.  Note that $H_i$ and $V_j$ are divisors on $\popo$,
defined by forms in $k[x_0,x_1,y_0,y_1]$ of degrees $(1,0)$ and $(0,1)$, respectively.
In the sequel, we shall abuse notation and let
$H_i$ and $V_j$ denote both the divisor and the form that defines it.

\begin{lemma} [{\cite[Theorem 5.1]{VT2}}]\label{gensofideal}
Let $X \subseteq \popo$ be a finite reduced ACM subscheme
with associated partition
$\lambda = (\lambda_1,\ldots,\lambda_h)$.  Let $H_1,\ldots,H_h$
denote the associated horizontal rules and $V_1,\ldots,V_v$ denote the
associated vertical rules which minimally contain $X$
(i.e., $H_i\cap X\neq\varnothing$ and $V_j\cap X\neq\varnothing$ for all $i$ and $j$).
Then a minimal homogeneous set of generators of $I(X)$ is given by
\[\{H_1\cdots H_h,V_1\cdots V_v\} \cup \{H_1\cdots H_{i}V_1 \cdots V_{\lambda_{i+1}} ~|~
\lambda_{i+1}-\lambda_i < 0 \}.\]
\end{lemma}

\begin{example} In Example \ref{renumber}, we have
$\lambda = (6,5,3,1,1)$, so $I(X)$ has generators
\[\{H_1H_2H_3H_4H_5,V_1V_2V_3V_4V_5V_6,H_1V_1V_2V_3V_4V_5,H_1H_2V_1V_2V_3,H_1H_2H_3V_1\}.\]
\end{example}

In light of Lemma \ref{gensofideal}, we can write down the generators
of $I(X)^2$ for any finite reduced ACM subscheme $X$ in $\popo$.
We end this section by showing that $I(X)^2 = I(X)^{(2)}$.  Note
that the ideal $I(X)^{(2)}$ defines a subscheme whose
support is $X$ and whose points all have multiplicity two
(alternatively, when viewed as a graded ideal,
 $I(X)^{(2)}$ defines a union of ``fat lines'' in $\mathbb{P}^3$).
We first recall a relevant fact. For our purposes, it
is sufficient to know that the algorithm described in \cite{GVT1}
always produces a set of generators for $I(X)^{(2)}$ of the following form.

\begin{lemma} \label{gensofideal2}
Let $X \subseteq \popo$ be a finite reduced ACM subscheme.
Let $H_1,\ldots,H_h$ denote the horizontal rules and
$V_1,\ldots,V_v$ denote the vertical rules which
minimally contain $X$. There is a minimal set of generators
of $I(X)^{(2)}$ such that every generator $F$ has one of the following forms:
$(a)$ $H_1^2\cdots H_h^2$; $(b)$ $H_1\cdots H_hV_1\cdots V_v$;
$(c)$ $V_1^2\cdots V_v^2$; or
$(d)$ there exist $1 \leq a \leq b \leq h$ and $1 \leq c \leq d \leq v$
such that
\[F = H_1^2H_2^2\cdots H_a^2H_{a+1}^1\cdots H_b^1V_1^2V_2^2\cdots V_c^2V_{c+1}^1\cdots
V_d^1.\]
(If $a = b$ in case $(d)$, $F$ has the form  $H_1^2H_2^2\cdots H_a^2V_1^2V_2^2\cdots V_c^2V_{c+1}^1\cdots
V_d^1$, and similarly for $c=d$.)
\end{lemma}

\begin{proof}[Sketch of the proof]
For a minimal set of generators for
$I(X)^{(2)}$ in terms of the partition $\lambda$, see
\cite[Theorem 3.15]{Gu2} and \cite[Algorithm 5.1]{GVT1}; in particular,
explicit values of $a$,$b$,$c$, and $d$ are given for the
elements described in $(d)$.
\end{proof}

\begin{theorem}\label{ACMsqrsThm}
Let $X \subseteq \popo$ be a finite reduced ACM subscheme.  Then
$I(X)^2 = I(X)^{(2)}$.
\end{theorem}

\begin{proof}
Let $I = I(X)$.  It suffices to prove $I^{(2)} \subseteq I^2$.
Let $F$ be any generator of $I^{(2)}$.
By Lemma \ref{gensofideal2}, $F$ must have one of four forms.
Since $H_1\cdots H_h$ and $V_1\cdots V_v$ are generators of $I$,
the generators
 $H_1^2\cdots H_h^2$, $H_1\cdots H_hV_1\cdots V_v$,
and $V_1^2\cdots V_v^2$ all belong to $I^2$.

We therefore take a generator of $I^{(2)}$ of the form
\[F = H_1^2H_2^2\cdots H_a^2H_{a+1}^1\cdots H_b^1V_1^2V_2^2\cdots V_c^2V_{c+1}^1\cdots
V_d^1\]
for some $1 \leq a \leq b \leq h$ and $1 \leq c \leq d \leq v$.  Factor
$F$ as
\[F_1 = H_1H_2\cdots H_aV_1\cdots V_cV_{c+1}\cdots V_d ~~\mbox{and}~~
F_2 = H_1H_2 \cdots H_aH_{a+1}\cdots H_bV_1\cdots V_c.\]
We claim that both $F_1$ and $F_2$ are elements of $I$ (and hence $F=F_1F_2\in I^2$).
We show that $F_1 \in I$, since the other case is similar.  Let $(P,Q)\in X$.
Then $(P,Q)$ is either on one of the rulings $H_1,\ldots,H_a$
or it is not.  If it is on one of these rulings, say $H_i$,
then $F_1((P,Q)) = 0$ because $H_i((P,Q)) = 0$.  On the other
hand, suppose that $(P,Q)$ is not on any of these rulings.  Since
$F$ vanishes with multiplicity two at $(P,Q)$, and
because $(P,Q)$ can lie on at most one of the rulings $H_{a+1},
\ldots,H_b$, there must be at least one
vertical ruling $V_j$ among $V_1,\ldots,V_d$ such that $V_j((P,Q)) = 0$.
But this means $F_1((P,Q)) = 0$.  Hence $F_1 \in I$.
\end{proof}


\section{Main Result}

We now present the proof of the main result of this paper, Theorem \ref{ACMresult}.

\begin{proof}[{Proof of Theorem \ref{ACMresult}}]
Since $I$ is homogeneous, we have $I^{(m)}=I^m$ if and only if $J^{(m)}=J^m$,
where $J=IR_M$, $R_M$ being the localization of $R=k[\popo]$ at the ideal
$M$ generated by the variables.
Note that $J$ is a perfect ideal
(i.e., $\operatorname{pd}_{R_M}(R_M/J)=\operatorname{depth}(J,R_M)$;
we have $\operatorname{depth}(J,R_M)= \operatorname{codim}(J)=2$
since $R_M$ is Cohen-Macaulay,
and we obtain $\operatorname{pd}_{R_M}(R_M/J)=2$ from the
Auslander-Buchsbaum formula),
it has codimension 2, $R_M/J$ is Cohen-Macaulay, and $J$ is generically
a complete intersection (i.e., the localizations of $J$ at its minimal associated primes
are complete intersections).
Now \cite[Theorem 3.2]{Mo} asserts that $J^{(m)}=J^m$ for all $m\geq 1$ if and only if
$J^{(m)}=J^m$ for $1\leq m\leq \dim(R_M)-1$.
Because $\dim(R_M) = 4$, it follows that
$J^{(2)}=J^2$ and $J^{(3)}=J^3$ implies $J^{(m)}=J^m$ for all $m \geq 1$,
and thus $I^{(2)}=I^2$ and $I^{(3)}=I^3$ implies $I^{(m)}=I^m$ for all $m \geq 1$.
But we always have $I^{(2)}=I^2$ by Theorem \ref{ACMsqrsThm},
so the conclusion follows.
\end{proof}

The next example shows that $I^{(3)}\neq I^3$
can occur for an ideal $I$ of a finite reduced ACM subscheme in $\popo$, while the example after that
shows that $I^{(3)}=I^3$ can occur for a finite reduced ACM subscheme,
even if it is not a complete intersection.

\begin{example}\label{sixptsexample}  Let $X$ be the reduced subscheme
consisting of six points
(unique up to choice of bi-homogeneous coordinates on $\popo$)
having diagram $S_X$ as in Figure 3.
Let $I=I(X)$ be the ideal of $X$ and let $\alpha(I) = \min\{d~|~
I_d \neq (0)\}$, where $I_d$ is the homogeneous component of $I$
of degree $d$ (with respect to the usual grading on $R=k[\pr^3]$).
It is easy to check that $\alpha(I)=3$ and hence
$\alpha(I^3)=3\alpha(I)=9$, and
so the bi-homogeneous component $(I^3)_{(4,4)}$ of bidegree $(4,4)$
is $(0)$ (since $(I^3)_{(4,4)}\subseteq (I^3)_8=(0)$). But there is a curve
$C\subset\popo$ of bidegree $(4,4)$
vanishing on $X$ to order 3.
It is, as shown in Figure 4,
the zero locus of $H_1^2V_1^2H_2V_2F$, where $\deg F = (1,1)$,
with the zero locus of $F$ represented by the diagonal line.
(One can check by B\'ezout's Theorem that in fact $C$ is unique.)

\hspace{.5in}\hbox to 3in
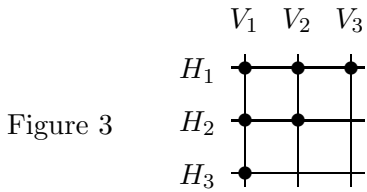
{
\begin{picture}(150,75)(25,5)
\put(-30,30){Figure 3}
\put(60,10){\line(0,1){50}}
\put(80,10){\line(0,1){50}}
\put(100,10){\line(0,1){50}}

\put(54,70){$V_1$}
\put(74,70){$V_2$}
\put(94,70){$V_3$}

\put(55,15){\line(1,0){50}}
\put(55,35){\line(1,0){50}}
\put(55,55){\line(1,0){50}}

\put(35,11){$H_3$}
\put(35,31){$H_2$}
\put(35,51){$H_1$}

\put(60,15){\circle*{5}} \put(63,18){$$}
\put(60,35){\circle*{5}} \put(63,38){$$}
\put(60,55){\circle*{5}} \put(63,58){$$}

\put(80,55){\circle*{5}} \put(83,58){$$}
\put(80,35){\circle*{5}} \put(83,38){$$}

\put(100,55){\circle*{5}} \put(103,58){$$}

\end{picture}
}

\vskip-.9in
\hspace{3.75in}\hbox to 3in
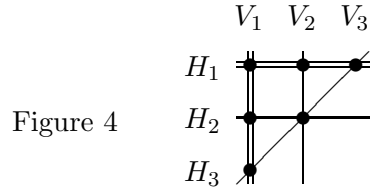
{
\begin{picture}(150,60)(25,5)
\put(-30,30){Figure 4}
\put(54,70){$V_1$}
\put(74,70){$V_2$}
\put(94,70){$V_3$}
\put(35,11){$H_3$}
\put(35,31){$H_2$}
\put(35,51){$H_1$}

\put(59,10){\line(0,1){50}}
\put(61,10){\line(0,1){50}}
\put(80,10){\line(0,1){50}}

\put(55,35){\line(1,0){50}}
\put(55,56){\line(1,0){50}}
\put(55,54){\line(1,0){50}}

\put(55,10){\line(1,1){50}}

\put(60,15){\circle*{5}} \put(63,18){$$}
\put(60,35){\circle*{5}} \put(63,38){$$}
\put(60,55){\circle*{5}} \put(63,58){$$}

\put(80,55){\circle*{5}} \put(83,58){$$}
\put(80,35){\circle*{5}} \put(83,38){$$}

\put(100,55){\circle*{5}} \put(103,58){$$}
\end{picture}
}
So, $(I^{(3)})_{(4,4)} \neq (0)$, whence $I^{(3)} \neq I^3$.
As pointed out in the introduction, this example gives a negative answer to the question of
Huneke discussed in the opening section.
\end{example}

In the next example we consider the case
(unique up to choice of bi-homogeneous coordinates on $\popo$) of
a finite reduced ACM subscheme consisting of 3 points which is not a complete intersection, but
whose ideal $I$ nevertheless has $I^{(3)} = I^3$.

\begin{example}\label{3acm}
Consider a reduced ACM subscheme $X$
consisting of 3 points, $Q_1,Q_2,Q_3\in\popo$, not
all on a single rule.   Up to
choice of coordinates, we may assume $P_1 = [1:0], P_2 =
[0:1]\in\pr^1$ and that $Q_1=(P_1,P_1), Q_2=(P_1,P_2),
Q_3=(P_2,P_1)$. In this case, the ideal $I=I(X)$ of $X$ is the
monomial ideal $I=(x_1,y_1) \cap (x_1,y_0) \cap (x_0,y_1)=(x_0x_1,
y_0y_1,x_0y_0)$. Then
\[I^3=(y_0^3y_1^3, x_0y_0^3y_1^2, x_0x_1y_0^2y_1^2,
x_0^2y_0^3y_1, x_0^2x_1y_0^2y_1, x_0^2x_1^2y_0y_1, x_0^3y_0^3,
x_0^3x_1y_0^2, x_0^3x_1^2y_0, x_0^3x_1^3).
\]
A tedious but elementary computation shows that
\[
\begin{split}
I^{(3)}&=(x_1,y_1)^3 \cap (x_1,y_0)^3 \cap (x_0,y_1)^3\\
&=(y_0^3y_1^3, x_0y_0^3y_1^2, x_0x_1y_0^2y_1^2, x_0^2y_0^3y_1,
x_0^2x_1y_0^2y_1, x_0^2x_1^2y_0y_1, x_0^3y_0^3, x_0^3x_1y_0^2,
x_0^3x_1^2y_0, x_0^3x_1^3).
\end{split}
\]
Thus $I^3=I^{(3)}$, and hence $I^m=I^{(m)}$ for all $m\geq1$ by
Theorem \ref{ACMresult}.
\end{example}


\section{Numerical conditions that imply $I^{(3)} = I^3$}

Let $I$ be the ideal of a reduced finite ACM subscheme $X$ in $\popo$. Theorem
\ref{ACMresult} reduces the problem of determining whether $I^{(m)} = I^m$
for all $m \geq 1$ to simply checking if equality holds when $m=3$.
It has been shown that many of the algebraic invariants of $I$ (e.g.,
the graded Betti numbers of its bigraded minimal free
resolution \cite{VT2}) are encoded into
its associated partition $\lambda$. This motivates us to ask if knowing $\lambda$
is enough to determine whether $I^{(3)} = I^3$, and consequently,
whether $I^{(m)} = I^m$ for all $m \geq 1$.
Computer experimentation using \cite{C,Mt} has suggested
the following specific characterization.

\begin{conjecture}\label{conjecture}
Let $X$ be a finite reduced ACM subscheme
in $\popo$ with associated partition
$\lambda = (\lambda_1,\ldots,\lambda_h)$.  Let $I = I(X)$.
Then $I^{(3)} = I^3$ if and only if $\lambda$ has one of the following
two forms:
\begin{enumerate}
\item[$(i)$] $\lambda = (a,a,\ldots,a)$ for some integer $a \geq 1$.
\item[$(ii)$] $\lambda = (a,\ldots,a,b,\ldots,b)$ for some integers $a > b \geq 1$.
\end{enumerate}
\end{conjecture}

In other words, the conjecture is that $I^{(3)} \neq I^3$ if and only
if $\lambda$ contains at least three distinct
entries.  Notice that this is the case in
Example \ref{sixptsexample} because the associated tuple is $\lambda
= (3,2,1)$.  We round out this paper by giving supporting evidence
for this conjecture.

\subsection{$\lambda$ has at most two distinct entries}
We first focus on the case that $\lambda = (a,\ldots,a)$.
If $X$ is a finite reduced ACM subscheme in $\popo$ with associated tuple
$\lambda = (a,\ldots,a)$, then  by Lemma \ref{gensofideal}, we have $I= I(X) =
(H_1\cdots H_{|\lambda|},V_1\cdots V_a)$, and so
$I$ is a complete
intersection.  Because $I$ is a
complete intersection, we have $I^{(m)}  = I^m$ for all $m$, and
in particular, for $m=3$. Thus, the $\lambda$'s of the
form $(i)$ in Conjecture \ref{conjecture} have the desired
property.

We therefore turn our attention to the case that $\lambda =
(a,\ldots,a,b,\ldots,b)$, and give evidence for the conjecture
by proving it when $\lambda = (\underbrace{a,\ldots,a}_t,a-1)$.
In this case
$$I(X) = (H_1\cdots H_{t+1},H_1\cdots H_tV_1\cdots V_{a-1},V_1\cdots V_a).$$
We first construct two new zero-dimensional subschemes $Y$ and $W$
such that $Z \subseteq W \subseteq Y$ where $Z$ is the
subscheme defined by $I(X)^{(3)}$.
Let $C$ denote the complete intersection defined by
$I(C) =  (H_1\cdots H_{t_1+1},V_1\cdots V_a)$.  Note that $X \subseteq C$.
We then define $Y$, respectively $W$, to be
the schemes defined by the ideals
\begin{equation}\label{schemesYW}
I(Y) =I(X)^{(3)} \cap\!\!\! \bigcap_{P\in C\setminus X}\!\!\! I(P)^2 ~~\mbox{and}~~~
I(W) = I(X)^{(3)} \cap\!\!\! \bigcap_{P\in C\setminus X}\!\!\! I(P).
\end{equation}
Note that in this case, $C\setminus X$ consists of a single point,
namely the point $P = H_{t+1} \cap V_a$.
Our strategy is to first find the generators of $I(Y)$, then find
the generators of $I(W)$ in terms of $I(Y)$, and then find the generators
of $I(X)^{(3)}$ in terms of the generators of $I(W)$.

As in \cite{Gu2}, we call $Y$ the {\it completion} of $Z$.
Moreover, by \cite{GVT}, we have:

\begin{lemma} \label{gensYspecial} Let  $X$ be a finite reduced ACM
subscheme in $\popo$ with partition $\lambda =
(\underbrace{a,\ldots,a}_{t},a-1)$. 
Let $Z$ denote the subscheme defined by $I(X)^{(3)}$
(i.e., $Z$ is the scheme whose points all have multiplicity three and whose
support is $X$), and let $Y$
denote the completion of $Z$ as described above.   Then a minimal set 
of generators of
$I(Y)$ is given by
\begin{eqnarray*}
G_1 &= & V_1^3\cdots V_{a}^3, \\
G_2 &= &H_1\cdots H_{t}V_{1}^3\cdots V_{a-1}^3V_{a}^2, \\
G_3 &= &H_1 \cdots H_{t+1}V_{1}^2\cdots V_{a}^2, \\
G_4 &= &H_1^2\cdots H_{t}^2H_{t+1}V_{1}^2\cdots V_{a-1}^2V_{a},\\
G_5 &= &H_1^2\cdots H_{t}^2H_{t+1}^2V_{1}\cdots V_{a}, \\
G_6 & = &H_1^3\cdots H_{t}^3H_{t+1}^2V_{1}\cdots V_{a-1}, ~~\text{and} \\
G_7 & = &H_1^3\cdots H_{t_1+1}^3
\end{eqnarray*}
\end{lemma}

\begin{proof}  For the reader's convenience, we sketch the main ideas.
As described at the beginning of section 3 of \cite{GVT},
we can associate to $Y$ two tuples $\alpha_Y$ and $\beta_Y$.
In our case, the tuples $\alpha_Y$ and $\beta_Y$ will satisfy the condition
of \cite[Theorem 4.8]{GVT} for $Y$ to be Cohen-Macaulay.  As a consequence,
\cite[Theorem 4.11]{GVT} gives us the degrees of the minimal
generators of $I(Y)$ directly from $\alpha_Y$.  Each of the forms listed
in the lemma belong to $I(Y)$ since they vanish at each point
with the correct multiplicity.  Furthermore, their degrees correspond
to the degrees of the minimal generators of $I(Y)$, so they form
our minimal set of generators.
\end{proof}

\begin{lemma} \label{gensWandXspecial}
Let $X$ be a finite reduced ACM subscheme in $\popo$ with
partition $\lambda = (\underbrace{a,\ldots,a}_{t},a-1)$.
Let $Z$ denote the subscheme defined by $I(X)^{(3)}$, and let $Y$ and
$W$ be the schemes defined as above. Then
\[I(W) = I(Y) + (H_1^2\cdots H_{t}^2 V_1^3 \cdots V_{a-1}^3V_{a},
H_1^3\cdots H_{t}^3H_{t+1}V_1^2 \cdots V_{a-1}^2),\]
and
\[I(X)^{(3)} = I(Z) = I(W) + (H_1^3\cdots H_{t}^3V_1^3\cdots V_{a-1}^3).\]
\end{lemma}

To prove this lemma, we introduce the notion of a separator:

\begin{definition}
Let $Z = m_1P_1 + \cdots + m_iP_i + \cdots + m_sP_s$ be the subscheme in $\popo$
defined by the ideal $\cap_j I(P_j)^{m_j}$.  We say that $F$ is a {\it separator of the
point $P_i$ of multiplicity $m_i$} if $F \in I(P_i)^{m_i-1}
\setminus I(P_i)^{m_i}$ and $F \in I(P_j)^{m_j}$ for all $j \neq
i$.  A set $\{F_1,\ldots,F_p\}$ is a set of {\it minimal separators of
$P_i$ of multiplicity $m_i$} if $I(Z')/I(Z) =
(\overline{F}_1,\ldots,\overline{F}_p)$, and there does not exist
a set $\{G_1,\ldots,G_q\}$ with $q < p$ such that  $I(Z')/I(Z) =
(\overline{G}_1,\ldots,\overline{G}_q)$.  Here,
$Z' =  m_1P_1 + \cdots + (m_i-1)P_i + \cdots + m_sP_s$.
\end{definition}

We now return to the proof of Lemma \ref{gensWandXspecial}.

\begin{proof}
Let $P$ denote the point given by $H_{t+1} \cap V_{a}$.  In
particular, the only point of $Y$ which is a double point is the point
whose support is $P$.  Because $Y$ is ACM, we can use
\cite{GVT2} to compute the minimal separators $P$ of multiplicity $2$.
Indeed, applying \cite[Theorem 3.4]{GVT2}, we find that
$H_1^2\cdots H_{t}^2 V_1^3 \cdots V_{a-1}^3V_{a}$ and
$H_1^3\cdots H_{t}^3H_{t+1}V_1^2 \cdots V_{a-1}^2$ are these minimal separators,
thus proving the first statement about the generators of $I(W)$.

We now wish to show that $I(W)+ (H_1^3\cdots H_{t}^3V_1^3\cdots V_{a-1}^3)
= I(X)^{(3)}$.  Note that $F =H_1^3\cdots H_{t}^3V_1^3\cdots V_{a-1}^3$ is a
separator of $P$ of multiplicity $1$.  To complete the proof,
we need to show that this form is the only minimal separator
of $P$ of multiplicity $1$.  So, suppose that $G$ is some
other separator of $P$ of multiplicity $1$.

Let $W_1$ consist
of the subscheme of $W$ which contains all the points on the ruling
$H_{t+1}$, and let $W_2$ consist of all the points on the ruling $V_a$.
So, $W_1$ consists of $a-1$ triple points and one reduced point, while
$W_2$ contains $t$ triple points and one reduce point.  The separator
$G$ is also a separator of this point of multiplicity $1$ in
both these schemes.

Now, $W_1$, respectively $W_2$, is ACM, so by \cite[Theorem 3.4]{GVT2},
it will have a unique, up to scalar,
minimal separator of degree $(0,3(a-1))$, respectively, $(3t,0)$.  Because
$G$ is a separator of $P$ of multiplicity $1$ in both these
schemes, we must have $\deg G \succeq (0,3(a-1))$ and $\deg G\succeq (3t,0)$.
So, $\deg G \succeq (3t,3(a-1))$.  So, $\deg G \succeq \deg F$.
But because $\deg Z = \deg W -1$, we must have that
$\dim_k (I(W)/I(Z))_{i,j} = 0$ or $1$ for all $(i,j)$.    We have shown
that if $G \in I(W) \setminus I(Z)$, we must have $\deg G \succeq \deg F
= (3t,3(a-1))$.  Thus we have $(I(W)/I(Z))$ is principally generated
by $\overline{F}$ in $R/I(Z)$, i.e., $F$ is the only
minimal separator of $P$ of multiplicity $1$.  This gives the desired
conclusion.
\end{proof}

We can now prove the following special case of Conjecture \ref{conjecture}.

\begin{theorem}\label{examplesofI^3}
Let $X$ be a finite reduced ACM subscheme in $\popo$.  Suppose
that the partition $\lambda$ associated to $X$ has the form
$\lambda = (\underbrace{a,a,\ldots,a}_t,a-1)$. Then $I(X)^{(3)} =
I(X)^3$.
\end{theorem}

\begin{proof}
The generators of $I(X)^{(3)}$ are given in
Lemmas \ref{gensYspecial} and \ref{gensWandXspecial}; an
exercise shows that each generator
belongs to $I(X)^3 =(H_1\cdots H_{t+1},H_1\cdots H_tV_1\cdots V_{a-1},V_1\cdots V_a)^3$.
\end{proof}

\subsection{$\lambda$ has at least three distinct entries}
One strategy to show that the converse of Conjecture \ref{conjecture}
holds is
to show that if  $\lambda$ has three or more distinct entries, then
$I(X)^{(3)} \neq I(X)^3$.  While we have not been able to prove this statement
in general, we conclude with some infinite families that exhibit this
behavior.

\begin{theorem}\label{nonacmexample}
Let $X$ be a finite reduced ACM subscheme in $\popo$.  Suppose
that the partition $\lambda$ associated to $X$ has
either of the two forms:
\begin{enumerate}
\item[$(i)$] $ \lambda = (\lambda_1,\ldots,\lambda_{t-3},3,2,1)
~~\mbox{with $\lambda_i \geq t-i+1$ for $i=1\ldots,t-3$}$; or
\item[$(ii)$] $\lambda = (\underbrace{t,t,\ldots,t}_m,t-1,t-2,\ldots,5,4,3,2,1) ~~\mbox{
for some integer $t \geq 3$}.$
\end{enumerate}
Then $I(X)^{(3)} \neq I(X)^3$.
\end{theorem}

\begin{proof}
Suppose that $\lambda$ has the form in $(i)$.
Consider the bi-homogeneous form
$$F = H_1^3H_2^3\cdots H_{t-3}^3H_{t-2}^2H_{t-1}^1V_1^2V_2D$$
where $H_i$ and $V_j$ correspond to the horizontal and vertical rulings that
minimally contain $X$, and $D$ is the degree $(1,1)$-form that vanishes at
the three points $H_{t-2} \cap V_3, H_{t-1} \cap V_2$, and $H_t \cap V_1$.
Then $F \in I(X)^{(3)}$ since $F$ vanishes at each of the points with
multiplicity at least three.  In addition, the bidegree of $F$
is $(3t-5,4)$, so the total degree of $F$ is $3t-1$.

By Lemma \ref{gensofideal}, the
bidegrees of the generators of $I(X)$
are $(t,0), (0,\lambda_1)$ and $(i,\lambda_{i+1})$ whenever
$\lambda_{i+1} - \lambda_i < 0$.  But by our hypotheses,
$\lambda_i \geq t - i +1$ for all $i$, thus all generators of $I(X)$
have total degree at least $t$, so $I(X)^3$ has no nonzero elements
of total degree less than $3t$, whence $F \in I(X)^{(3)} \setminus I(X)^3$.

Now assume $(ii)$.
The form
$F = H_1^3H_2^3\cdots H_m^3H_{m+1}^3 \cdots H_{m+t-4}^3H_{m+t-3}^2H_{m+t-2}V_1^2V_2D$
vanishes at each of the points with multiplicity at least three;  again
$D$ is a $(1,1)$-form that vanishes at
the three points $H_{m+t-3} \cap V_3, H_{m+t-2} \cap V_2$, and $H_{m+t-1} \cap V_1$.
The form $F$ has bidegree $(3m+3t-8,4)$.

By Lemma \ref{gensofideal},
$I(X)$ has $t+1$ generators, say $G_0,\ldots,G_t$
with $\deg G_i = (t+m-i,i)$ for $i = 0,\ldots,t-1$, and $\deg G_t = (0,t)$.
If $F \in I(X)^3$, then there exists some non-negative integer solution
to $a_0 + a_1 + \cdots + a_t = 3$ such that
$(3m+3t-8,4) \succeq \deg G_0^{a_0}G_1^{a_1}\cdots G_t^{a_t}$.
This expression implies that our non-negative integer solution to
$a_0 + a_1 + \cdots + a_t = 3$ must also satisfy
\[0a_0 +1a_1 + 2a_2 + \cdots + ta_t\leq 4.\]
We can write out all such solutions:
\begin{eqnarray*}
(3,0,0,0,0,0,\ldots,0), & (0,3,0,0,0,0,\ldots,0), & (2,1,0,0,0,0,\ldots,0), \\
(2,0,1,0,0,0,\ldots,0), & (2,0,0,1,0,0,\ldots,0), & (2,0,0,0,1,0,\ldots,0), \\
(1,2,0,0,0,0,\ldots,0), & (1,0,2,0,0,0,\ldots,0), & (1,1,1,0,0,0,\ldots,0), \\
& (1,1,0,1,0,0,\ldots,0).&
\end{eqnarray*}
However, for any such solution, $(3m+3t-8,4) \not\succeq
\deg G_0^{a_0}G_1^{a_1}\cdots G_t^{a_t}$ because the first coordinate of
$\deg G_0^{a_0}G_1^{a_1}\cdots G_t^{a_t}$ will be larger than $3m+3t-8$.
So, $I(X)^3$ will be empty in bidegree $(3m+3t-8,4)$, but $I(X)^{(3)}$
is not empty.  This implies the desired conclusion.
\end{proof}

\begin{remark} Theorem \ref{nonacmexample}
generalizes Example \ref{sixptsexample},
which has $\lambda =(3,2,1)$.
\end{remark}



\begin{thebibliography}{99}

\bibitem{BCH}
C.\ Bocci, S.\ Cooper and B.\ Harbourne,
Containment results for ideals of various configurations of points in $\pr^N$.
Preprint (2011).
{\tt arXiv:1109.1884v1}

\bibitem{BH} C.\ Bocci and B.\ Harbourne,
Comparing powers and symbolic power of ideals.
J.\ Algebraic Geom. {\bf 19} (2010), 399--417.

\bibitem{BH2}
C.\ Bocci and B.\ Harbourne,
The resurgence of ideals of points and the containment problem.
Proc.\ Amer.\ Math.\ Soc.  {\bf 138}  (2010), 1175--1190.

\bibitem{C}
CoCoATeam, CoCoA: a system for doing Computations
in Commutative Algebra.
Available at \url{http://cocoa.dima.unige.it}

\bibitem{DJ}
A.\ Denkert and M.\ Janssen,
Containment problem for points on a reducible conic in $\mathbb{P}^2$.
Preprint (2012).
{\tt arXiv:1207.7153v1}

\bibitem{ELS} L.\ Ein, R.\ Lazarsfeld, and K.\ E.\ Smith,
Uniform behavior of symbolic powers of ideals.
Invent.\ Math. {\bf 144} (2001), 241--252.

\bibitem{GHM} A.\ V.\ Geramita, B.\ Harbourne and J.\ Migliore,
Star configurations in $\mathbb{P}^n$.
Preprint (2012).
{\tt arXiv:1203.5685v1}


\bibitem{GuMaRa} S.\ Giuffrida, R.\ Maggioni and A.\ Ragusa,
On the postulation of $0$-dimensional subschemes on a smooth quadric.
Pacific J.\ Math. {\bf 155} (1992), 251--282.

\bibitem{Mt} D.\ R.\ Grayson and M.\ E.\ Stillman,
Macaulay 2, a software system for research in algebraic geometry.
\url{http://www.math.uiuc.edu/Macaulay2/}.

\bibitem{GHVT} E.\ Guardo, B.\ Harbourne and A.\ Van Tuyl,
Symbolic powers versus regular powers of ideals of general points
in $\popo$.
Preprint (2011).
{\tt arXiv:1107.4906v2}

\bibitem{GHVT2} E.\ Guardo, B.\ Harbourne and A.\ Van Tuyl,
Asymptotic resurgences for ideals of positive dimensional
subschemes of projective space.
Preprint (2012). {\tt arXiv:1202.4370v1}

\bibitem{Gu2} E.\ Guardo, Fat point schemes on a smooth quadric.
J.\ Pure Appl.\ Algebra {\bf 162} (2001), 183--208.

\bibitem{GVT} E.\ Guardo and A.\ Van Tuyl,
Fat points in $\popo$ and their Hilbert functions.
Canad.\ J.\ Math. {\bf 56} (2004), 716--741.

\bibitem{GVT1} E.\ Guardo and A.\ Van Tuyl,
The minimal resolution of double points in $\popo$ with ACM support.
J.\ Pure Appl.\ Algebra {\bf 211} (2007), 784--800.

\bibitem{GVTcollect}
E.\ Guardo and A.\ Van Tuyl,
ACM sets of points in multiprojective space.
Collect. Math. {\bf 59} (2008), 191--213.

\bibitem{GVT2} E.\ Guardo and A.\ Van Tuyl,
Separators of Arithmetically Cohen Macaulay fat points in $\popo$.
J. Commut. Alg.  {\bf 4} (2012), 255--268.


\bibitem{HaHu} B.\ Harbourne and C.\ Huneke,
Are symbolic powers highly evolved?
(2011) To appear J. Ramanujan Math. Soc. {\tt arXiv:1103.5809v1}

\bibitem{HH}
M.\ Hochster and C.\ Huneke,
Comparison of symbolic and ordinary powers of ideals.
Invent.\ Math. {\bf 147} (2002), 349--369.

\bibitem{HH2} M. Hochster and C. Huneke. Fine behavior of
symbolic powers of ideals.  Illinois J. Math.  {\bf 51}  (2007),  171--183.

\bibitem{LS}
A.\ Li and I.\ Swanson,
Symbolic powers of radical ideals,
Rocky Mountain J. Math. {\bf 36} (2006) 997--100.

\bibitem{Mo} S.\ Morey, Stability of associated primes and equality of ordinary
and symbolic powers of ideals.
Comm.\ Algebra {\bf 27} (1999), 3221--3231.

\bibitem{VT2} A.\ Van Tuyl,
The Hilbert functions of ACM sets of points in $\mathbb{P}\sp
{n\sb 1}\times\dots\times\mathbb{P}\sp {n\sb k}$.
J.\ Algebra {\bf 264} (2003), 420--441.

\bibitem{ZS}
O.\ Zariski and P.\ Samuel,
{\em Commutative algebra. Vol. II.}
The University Series in Higher Mathematics, D. Van Nostrand Co., Inc., Princeton,
N. J.-Toronto-London-New York, 1960.
\end{thebibliography}
\end{document}